\documentclass{amsart}

\usepackage{amssymb,amsmath,amsthm,color}

\usepackage[bookmarksopen=true,bookmarksnumbered=true, bookmarksopenlevel=2, colorlinks=true,linkcolor=mblue,
citecolor=mblue,urlcolor=mblue, pdfstartview=FitH]{hyperref}

\definecolor{mblue}{rgb}{0,0,.8}

\newcommand{\N}{\mathbb N}
\newcommand{\Z}{\mathbb Z}
\newcommand{\Q}{\mathbb Q}

\newcommand{\F}{\mathbb F}

\newcommand{\C}{\mathbb C}
\newcommand{\p}{\mathfrak p}

\newcommand{\Zmod}[1]{\overline{\Z/{#1}\Z}}
\newcommand{\Zmodpm}[1]{\Z/{#1}\Z}
\newcommand{\diam}[1]{\left< #1 \right>}

\newcommand{\abcd}[4]{\left(\begin{smallmatrix}#1&#2\\#3&#4\end{smallmatrix}\right)}

\newtheorem{thm}{Theorem}
\newtheorem{lem}{Lemma}
\newtheorem{remark}{Remark}

\newtheorem{prop}{Proposition}
\newtheorem{cor}{Corollary}

\DeclareMathOperator{\GL}{GL} \DeclareMathOperator{\SL}{SL}  
   
  \DeclareMathOperator{\Tr}{tr} 

\def\dash---{\thinspace---\hskip.16667em\relax}

\begin{document}

\title{On the theta operator for modular forms modulo prime powers}
\author[Imin Chen, Ian Kiming]{Imin Chen, Ian Kiming}
\address[Imin Chen]{Department of Mathematics, Simon Fraser University, 8888 University Drive, Burnaby, B.C., V5A 1S6, Canada}
\email{ichen@math.sfu.ca}
\address[Ian Kiming]{Department of Mathematical Sciences, University of Copenhagen, Universitetsparken 5, 2100 Copenhagen \O , Denmark}
\email{kiming@math.ku.dk}

\begin{abstract} We consider the classical theta operator $\theta$ on modular forms modulo $p^m$ and level $N$ prime to $p$ where $p$ is a prime greater than $3$. Our main result is that $\theta$ mod $p^m$ will map forms of weight $k$ to forms of weight $k+2+2p^{m-1}(p-1)$ and that this weight is optimal in certain cases when $m$ is at least $2$. Thus, the natural expectation that $\theta$ mod $p^m$ should map to weight $k+2+p^{m-1}(p-1)$ is shown to be false.

The primary motivation for this study is that application of the $\theta$ operator on eigenforms mod $p^m$ corresponds to twisting the attached Galois representations with the cyclotomic character. Our construction of the $\theta$-operator mod $p^m$ gives an explicit weight bound on the twist of a modular mod $p^m$ Galois representation by the cyclotomic character.
\end{abstract}

\maketitle

\section{Introduction} \label{intro} Let $p$ be a prime number. We shall assume $p\ge 5$ throughout the paper in order to avoid certain technicalities when $p$ is $2$ or $3$. Also, we fix a natural number $N$ that will be assumed throughout to be not divisible by $p$.

Let $\Z_p$ be the ring of $p$-adic integers, $\Q_p$ the field of $p$-adic numbers, and $\C_p$ the completion of an algebraic closure of $\Q_p$. Let $\overline{\Q}$ be an algebraic closure of $\Q$. We fix embeddings of $\overline{\Q}$ into $\C$ and $\C_p$ and use this to freely make comparisons of elements between these fields.

Further let $m \in \N$ and denote by $M_k(N,\Z_p)$ the $\Z_p$-module of modular forms of weight $k$ on $\Gamma_1(N)$ over $\Z_p$ and let $M_k(N,\Zmodpm{p^m})$ be the $\Z/p^m\Z$-module of modular forms on $\Gamma_1(N)$ over $\Z/p^m\Z$ as defined classically by $M_k(N,R) = M_k(N,\Z) \otimes R$. Note that this definition relies on the existence of an integral structure on $M_k(N,\C)$ (see for instance \cite{ckw}). Thus, as we are working with this classical interpretation of `modular forms with coefficients in $R$', we are able to switch effortlessly between forms with coefficients in $\Z$ or $\Z_p$, and those with coefficients in $\Z/p^m\Z$: for forms with coefficients in $\Z$ or $\Z_p$ reduction modulo $p^m$ gives a form with coefficients in $\Z/p^m\Z$, and any such can be lifted to a form with coefficients in $\Z$ (and hence also to one with coefficients in $\Z_p$). We also denote by $S_k(N,R)$ the module of cusp forms of weight $k$ on $\Gamma_1(N)$ and coefficients in $R$, again in the classical sense.

Let $k_1,\ldots,k_t$ be a collection of weights and let $f_i \in M_{k_i}(N,\Z_p)$. By `$q$-expansion' we shall always mean `$q$-expansion at $\infty$'. The $q$-expansion of an element in a direct sum of the $M_{k_i}(N,\Z_p)$'s or $M_{k_i}(N,\Zmodpm{p^m})$'s is defined by extending linearly on each component. When we write $f_1 + \ldots + f_t \equiv 0 \pmod{p^m}$, we shall mean that the $q$-expansion $f_1(q) + \ldots + f_t(q)$ lies in $p^m \Z_p[[q]]$. Similarly for $f_i \in M_{k_i}(N,\Zmodpm{p^m})$, we write $f_1 + \ldots + f_t \equiv 0 \pmod{p^m}$ if the $q$-expansion of $f_1 + \ldots + f_t$ equals $0$ in $(\Z/p^m\Z)[[q]]$. In such a case, we say that $f_1 + \ldots + f_t$ is congruent to $0$ modulo $p^m$.

Let us recall the definition and basic properties of the standard Eisenstein series on $\SL_2(\Z)$, cf.\ \S 1 of \cite{serre_zeta}, for instance: the series
$$
G_k := -\frac{B_k}{2k} + \sum_{n=1}^{\infty} \sigma_{k-1}(n) q^n
$$
with $B_k$ the $k$-th Bernoulli number and $\sigma_t(n) := \sum_{d\mid n} d^t$ the usual divisor sum, is (with $q:=e^{2\pi iz})$ for $k$ an even integer $\ge 4$ a modular form on $\SL_2(\Z)$. Defining $E_k$ as the normalization
$$
E_k := -\frac{2k}{B_k} \cdot G_k
$$
one has $E_k \equiv 1 \pmod{p^t}$ when (and only when) $k\equiv 0 \pmod{p^{t-1}(p-1)}$.

There are natural inclusions (preserving $q$-expansions)
\begin{align*}
  M_k(N,\Zmodpm{p^m}) \hookrightarrow M_{k+p^{m-1}(p-1)}(N,\Zmodpm{p^m}),
\end{align*}
induced by multiplication by $E_{p-1}^{p^{m-1}}$, using the fact that $E_{p-1}^{p^{m-1}} \equiv 1 \pmod {p^m}$. Note that $E_{p^{m-1}(p-1)} = E_{p-1}^{p^{m-1}}$ in $M_k(N,\Zmodpm{p^m})$ (as both forms have $q$-expansions which are congruent modulo $p^m$) so that the map can also be seen as induced by multiplication by $E_{p^{m-1}(p-1)}$.

As is well-known, when we specialize the above series for $G_k$ to $k=2$ and define
$$
G_2 := -\frac{B_2}{4} + \sum_{n=1}^{\infty} \sigma_1(n) q^n = -\frac{1}{24} + \sum_{n=1}^{\infty} \sigma_1(n) q^n
$$
then $G_2$ does not represent a modular form in the usual sense, but does so in the $p$-adic sense, cf.\ \cite{serre_zeta}, \S 2. One defines $E_2$ as the normalization of $G_2$, i.e., $E_2 := -24 G_2$. Thus, $E_2$ is also a $p$-adic modular form.

Consider the classical theta operator $\theta f = \frac{k}{12} E_2 + \frac{1}{12} \partial f$ of Ramanujan. Its effect on $q$-expansions is $\sum_n a_n q^n \mapsto \sum_n n a_n q^n$. Since $E_2$ is a $p$-adic modular form it is for any $m\in\N$ congruent modulo $p^m$ to a classical modular form of some weight. Thus we have $E_2 \equiv E_{p+1} \pmod{p}$, for example, and since we classically know that $\partial$ maps modular forms of weight $k$ to modular forms of weight $k+2$, one obtains the classical operator $\theta$ that maps $M_k(N,\F_p)$ to $M_{k+p+1}(N,\F_p)$. Studying this operator as well as its interaction with the `weight filtration' (see below) is a key tool in the theory of modular forms modulo $p$; cf.\ for instance Jochnowitz' proof of finiteness of systems of Hecke eigenvalues mod $p$ across all weights in \cite{jochnowitz}, or Edixhoven's results on the optimal weight in Serre's conjectures \cite{edix}.

We have launched a framework for the study of modular forms mod $p^m$ in \cite{ckr} and \cite{ckw}. Notice that Serre shows in \cite[Th\'{e}or\`{e}me 5]{serre_zeta} that there exists a $\theta$ operator on $p$-adic modular forms (of level $1$, however the arguments generalize to levels $N$ not divisible by $p$) whose effect on $q$-expansions is $\sum_n a_n q^n \mapsto \sum_n n a_n q^n$ and that sends a form of ($p$-adic) weight $k$ to a form of weight $k+2$. This result immediately implies the existence of a $\theta$ operator modulo $p^m$ for modular forms on $\Gamma_1(N)$.

We have discussed in \cite{ckw} how to attach Galois representations to eigenforms mod $p^m$, and it is clear from the properties of those attached representations that applying the $\theta$ operator corresponds on the Galois side to twisting by the cyclotomic character mod $p^m$. From this perspective, it seems natural to ask about the finer properties of the action of the $\theta$ operator modulo $p^m$, such as how the action of $\theta$ modulo $p^m$  changes weight filtrations.

Our results show that the interplay of the $\theta$ operator with the weights of the forms becomes much more complicated when $m>1$ and that, in fact, there are certain genuine qualitative differences between the case $m=1$ and the general cases $m>1$. Let us explain in detail.

We first show that the $\theta$ operator on modular forms mod $p^m$ maps
$$
M_k(N,\Zmodpm{p^m}) \rightarrow M_{k+k(m)}(N,\Zmodpm{p^m})
$$
where $k(m) := 2 + 2 p^{m-1}(p-1)$. The fact that the effect of $\theta$ on $q$-expansions is $\sum_n a_n q^n \mapsto \sum_n n a_n q^n$, and that $\theta$ satisfies simple commutation rules with Hecke operators $T_{\ell}$ for primes $\ell \neq p$, cf.\ the first two parts of Theorem \ref{theta-single-wt} below, follow already from Serre's theory of the $p$-adic $\theta$ operator, but we will give self-contained proofs working purely in the mod $p^m$ setting, and for modular forms on $\Gamma_1(N)$.

The proof of the above properties use a number of results from \cite{serre_zeta} plus the observation that $f \mid V \equiv f^p \pmod p$ where $V$ is the classical $V$ operator.

Define the weight $w_{p^m}(f)$ of a modular form $f$ mod $p^m$ with $f \not \equiv 0 \pmod p$ to be the smallest $k \in \Z$ such that $f$ is congruent modulo $p^m$ to an element of $M_k(N,\Zmodpm{p^m})$. A classical fact, crucial for instance in the work \cite{jochnowitz}, is that when $m=1$ we have $w_p(\theta f) \le w_p(f) + p + 1$ with equality if (and only if) $p \nmid w_p(f)$.

One might expect the generalization of this to be that $w_{p^m}(\theta f) \le w_{p^m}(f) + 2 + p^{m-1}(p-1)$ (perhaps with equality in some cases). However, as the third part of Theorem \ref{theta-single-wt} shows, this is false:

\begin{thm}\label{theta-single-wt} Let $p \ge 5$ be a prime. Put
$$
k(m) := 2 + 2 p^{m-1}(p-1).
$$

(i) The classical theta operator $\theta$ induces an operator
$$
\theta : M_k(N,\Zmodpm{p^m}) \rightarrow M_{k + k(m)}(N,\Zmodpm{p^m})
$$
whose effect on $q$-expansions is $\sum a_n q^n \mapsto \sum n a_n q^n$. The operator preserves cusp forms.
\smallskip

(ii) If $\ell \not= p$ is a prime and $T_{\ell}$ denotes the $\ell$-th Hecke operator, then
$$
T_{\ell} \theta = \ell \cdot \theta T_{\ell}
$$
as linear maps $M_k(N,\Zmodpm{p^m}) \rightarrow M_{k + k(m)}(N,\Zmodpm{p^m})$. Furthermore, if $\ell = p$ and $k \ge m$, we still have the identity above.
\smallskip

\noindent (iii) Let $m\ge 2$ and $f \in M_k(N,\Zmodpm{p^m})$ with $f \not\equiv 0 \pmod p$. Suppose further that $p\nmid k$ and $w_p(f)= k$. Then
$$
w_{p^m}(\theta f) = k + 2 + 2 p^{m-1} (p-1).
$$
\end{thm}

The proof of the third part of the Theorem uses a result of Katz: if we consider the Eisenstein series $E_{p-1}$ and $E_{p+1}$ as modular forms modulo $p$ in the sense of Katz then $E_{p-1}$ (Hasse invariant) has no zero in common with $E_{p+1}$; cf.\ \cite[Remark on p.\ 57]{katz_modp}, or \cite[Theorem 3.1]{katz_zannier}. This point allows one to compute the weight filtration of the last term of the expression below for the $\theta$ operator mod $p^m$, which is the controlling term. We note that, in the case $N=1$ and $m=2$, it is possible to give more complete results on the effect of $\theta \mod p^m$ on weight filtrations.

We would like to remark that the first two parts of Theorem \ref{theta-single-wt} hold with the same proofs even if one allows the level $N$ to be divisible by $p$. However, we have chosen to assume $p\nmid N$ throughout for the following reasons: first, we use this in the third and most important part of Theorem \ref{theta-single-wt}. Secondly, if one allows $p\mid N$ many of the questions we discuss can become easier; for instance, if $p^m \mid N$, then $\theta$ mod $p^m$ raises the weight by at most $2$ simply because $E_2$ is then congruent modulo $p^m$ to a modular form of weight $2$ on $\Gamma_0(p^m)$ (namely, $E_2(z) - p^m E_2(p^m z)$). Thirdly, in another sense matters can become more complicated if we allow $p$ to divide $N$. For instance, there may then exist a form modulo $p^m$ that is $0$ in our sense, i.e., with $p$-integral $q$-expansion around $\infty$ that is term-wise congruent to $0$ modulo $p^m$, but at the same time such that it has a $p$-integral expansion not identically $0$ modulo $p^m$ around some other cusp. For a concrete example of this phenomenon, consider the classical Eisenstein series
$$
F(z) = \sum_{n=1}^{\infty} \left( \sum_{d\mid n} \chi(n/d) d^2 \right) e^{2\pi inz}
$$
where $\chi$ is the Dirichlet character corresponding to $\Q(\sqrt{-1})/\Q$. Then $F$ is a weight $3$ form on $\Gamma_0(4)$ with nebentypus $\chi$. If now $p$ is an odd prime and we define the form $f$ by $f(z) := p^3 F(pz)$ then $f$ is a form of weight $3$ on $\Gamma_1(4p)$ and one can verify that $f$ has $p$-integral Fourier expansion around every cusp and that the expansion around the cusp $1$ begins with a $p$-adic unit. But of course, the expansion around $\infty$ is term-wise identically $0$ modulo $p^3$. (Details of the verification can be found in \cite{kiming_asymp}, proof of Theorem 1.)
\smallskip

For simplicity, we have stated results in this paper for modular forms with coefficients in $\Z_p$ and hence reductions with coefficients in $\Zmodpm{p^m}$. The above theorem however is valid for coefficients in $\Zmod{p^m}$ (see \cite[\S 2.2]{taixes-wiese} or \cite[Section 1]{ckw} for a definition of this ring) using essentially the same proofs.

An immediate consequence of Theorem~\ref{theta-single-wt} to Galois representations is the following. We use the notation and terminology from \cite{ckw}.
\begin{cor}
\label{main-cor}
Let $p \ge 5$ be a prime, $\rho : G_\Q \rightarrow \GL_2(\Zmod{p^m})$ be a residually absolutely irreducible Galois representation, and $\chi :G_\Q \rightarrow (\Zmodpm{p^m})^\times$ be the reduction modulo $p^m$ of the $p$-adic cyclotomic character. Suppose $\rho \cong \rho_f$ for some weak eigenform $f \in S_k(N)(\Zmod{p^m})$. Then $\rho \otimes \chi \cong \rho_{g}$ for some weak eigenform $g \in S_{k + k(m)}(N)(\Zmod{p^m})$.
\end{cor}
\begin{proof} By (i) of Theorem \ref{theta-single-wt}, the $\theta$ operator maps cusp forms to cusp forms. Suppose then that $f \in S_k(N)(\Zmod{p^m})$ is a weak eigenform in the sense that $T_\ell f \equiv f(T_\ell) f \pmod{p^m}$ for all $\ell \nmid D$ for some integer $D$. Then $g = \theta f \in S_{k+k(m)}(N)(\Zmod{p^m})$ is a weak eigenform in the sense that $T_\ell g \equiv g(T_\ell) g \pmod{p^m}$, for all $\ell \nmid Dp$, where $g(T_\ell) \equiv f(T_\ell) \chi(\ell) \pmod{p^m}$. The hypothesis of residual absolute irreducibility then allows us to conclude that $\rho_f \otimes \chi \cong \rho_{\theta f}$.
\end{proof}

Regarding Corollary~\ref{main-cor}, the main result of the paper \cite{cs} implies an analogous statement about twisting with the mod $p^m$ reduction of the Teichm\"uller character, with some differences: In \cite{cs}, they consider the twist of $f$ by the mod $p^m$ reduction of the Teichm\"uller character where $f$ is a strong eigenform (again using the terminology of \cite{ckw}). One then finds a strong (and not merely weak) eigenform $g$ as in the Corollary, but apparently without control over the weight $k+k(m)$. The proof uses different methods (Coleman $p$-adic families of modular forms.)

\section{The theta operator modulo prime powers}

\subsection{Eisenstein series}\label{eisenstein_series} We shall now develop an explicit expression for the truncation modulo $p^m$ of the $p$-adic Eisenstein series $G_2$.

Recall that if $f=\sum_n a_n q^n$ is nonzero modular form on some $\Gamma_1(N)$ with coefficients in $\Q_p$ then as usual one defines $v_p(f) := \min \{ v_p(a_n) \mid ~ n\in\N \}$ where $v_p(a_n)$ is the usual (normalized) $p$-adic valuation of $a_n$.

\begin{prop}\label{prop:G_2_modpm} Let $m\in\N$. Define the positive even integers $k_0,\ldots,k_{m-1}$ as follows: If $m\ge 2$, define:
$$
k_j := 2 + p^{m-j-1}(p^{j+1}-1) \quad \mbox{for $j=0,\ldots,m-2$}
$$
and
$$
k_{m-1} := p^{m-1}(p+1)
$$
and define just $k_0 := p+1$ if $m=1$.

Then $k_0 < \ldots < k_{m-1}$ and there are modular forms $f_0,\ldots,f_{m-1}$, depending only on $p$ and $m$, of level $1$ and of weights $k_0,\ldots,k_{m-1}$, respectively, that have rational $q$-expansions, satisfy $v_p(f_j) = 0$ for all $j$, and are such that
$$
G_2 \equiv \sum_{j=0}^{m-1} p^j f_j \pmod{p^m}
$$
as a congruence between $q$-expansions.

The form $f_{m-1}$ can be chosen to be $f_{m-1} = G_{p+1}^{p^{m-1}}$.
\end{prop}

When $m=2$ we can, and will, be a bit more explicit:

\begin{prop}\label{prop:G_2_modp2} We have:
$$
G_2 \equiv f_0 + p\cdot f_1 \pmod{p^2}
$$
with modular forms $f_0$ and $f_1$ of weights $2+p(p-1)$ and $p(p+1)$, respectively, explicitly:
$$
G_2 \equiv G_{2+p(p-1)} + p\cdot G_{p+1}^p \pmod{p^2} .
$$
\end{prop}

It is amusing to note the following consequence of the Proposition: For $p \not=2, 3$, we have the following congruence of Bernoulli numbers,
$$
\frac{B_2}{2} \equiv \frac{B_{p(p-1)+2}}{p(p-1)+2} + p \frac{B_{p+1}}{p+1} \pmod {p^2}.
$$

However, this can also be seen in terms of $p$-adic continuity of Bernoulli numbers (cf.\ \cite{washington}, Cor.\ 5.14 on p.\ 61, for instance).
\smallskip

Before the proofs of these propositions we need a couple of preparations.
\smallskip

Let $k$ be an even integer $\ge 2$. Recall from \cite{serre_zeta} that if we choose a sequence of even integers $k_i$ such that $k_i \rightarrow \infty$ in the usual, real metric, but $k_i\rightarrow k$ in the $p$-adic metric, then the sequence $G_{k_i}$ has a $p$-adic limit denoted by $G_k^{\ast}$. This series $G_k^{\ast}$ is a $p$-adic modular form of weight $k$ with $p$-adically integral coefficients if $(p-1) \nmid k$ (this condition ensures that the constant term is a $p$-adic integer). It does not depend on the choice of the sequence $k_i$.

\begin{lem}\label{lemma:G_k^*} Let $k$ be an even integer $\ge 2$ and assume $(p-1)\nmid k$. Let $t\in\N$.

Then $G_k^{\ast} \equiv G_{k+p^{t-1}(p-1)} \pmod{p^t}$.
\end{lem}
\begin{proof} Put $k_i := k + p^{i-1}(p-1)$ for $i\in\N$. We claim that $G_{k_u} \equiv G_{k_v} \pmod{p^t}$ when $u,v\ge t$. Since the series $G_{k_i}$ with $k_i := k + p^{i-1}(p-1)$ converges $p$-adically to $G_k^{\ast}$, the claim clearly implies the Lemma.

If $u=v$ the claim is trivial, so suppose not, say $u<v$. Then $k_v-k_u = p^{v-1}(p-1) - p^{u-1}(p-1)$ is a multiple of $p^{t-1}(p-1)$, say $k_v - k_u = s\cdot p^{t-1}(p-1)$. We also have $k_v - k_u \ge 4$. Hence, we find that $G:=G_{k_u} \cdot E_{p^{t-1}(p-1)}^s$ is a modular form of weight $k_v$, and we have $G_{k_u} \equiv G \pmod{p^t}$.

Now notice that, when $i\ge t$, we have:
$$
\sigma_{k_i-1}(n) = \sum_{d\mid n} d^{k-1+p^{i-1}(p-1)} \equiv \sum_{\substack{d\mid n \\ p\nmid d}} d^{k-1} \pmod{p^t}
$$
as $d^{p^{i-1}(p-1)} \equiv 1 \pmod{p^t}$ when $p\nmid d$ and $i\ge t$, and as $d^{p^{i-1}(p-1)} \equiv 0 \pmod{p^t}$ when $p\mid d$ and $i\ge t$ because $p^{t-1}(p-1) \ge t$ for $t\in\N$.

We conclude that the nonconstant terms of the series $G_{k_u}$ and $G_{k_v}$ are termwise congruent modulo $p^t$. The same is then true of the forms $G$ and $G_{k_v}$ that are both forms of weight $k_v$. Hence, the nonconstant terms of the form $(G-G_{k_v})/p^t$ are all $p$-integral. As $k_v \equiv k \not\equiv 0 \pmod{(p-1)}$, it follows from Th\'{e}or\`{e}me 8 of \cite{serre_modp} that the constant term of this form is in fact also $p$-integral. Hence,
$$
G_{k_u} \equiv G \equiv G_{k_v} \pmod{p^t}
$$
as desired.
\end{proof}

Recall that the $V$ operator is defined on formal $q$-expansions as
$$
\left( \sum a_n q^n \right) \mid V := \sum a_n q^{np} .
$$

\begin{cor}\label{cor:G_2a} We have
$$
G_2 \equiv \sum_{j=0}^{m-1} p^j \cdot \left( G_{2+p^{m-j-1}(p-1)}\mid V^j \right) \pmod{p^m}
$$
as a congruence between formal $q$-expansions.
\end{cor}
\begin{proof} Recall from \cite{serre_zeta}, \S 2, the identity, valid for any even integer $k\ge 2$, that
$$
G_k = G_k^{\ast} + p^{k-1} \left( G_k^{\ast} \mid V \right) + \ldots + p^{t(k-1)} \left( G_k^{\ast} \mid V^t \right) + \ldots .
$$

The identity can first be seen as an identity of formal $q$-expansions, but when we specialize to $k=2$ it shows that $G_2$ is a $p$-adic modular form as $V$ acts on $p$-adic modular forms, cf.\ \cite{serre_zeta}, \S $2$.

Thus, we consider the identity for $k=2$, reduce modulo $p^m$, and note that the previous Lemma applies since $(p-1)\nmid 2$. The claim then follows immediately.
\end{proof}

In the next paragraph and lemma, we use the notation $M_k(\Gamma,F)$ to mean the $F$-module of modular forms of weight $k$ over $F$, where $\Gamma$ is a congruence subgroup of $\SL_2(\Z)$ and $F$ is a subring of $\C$ or $\C_p$.

We can also see the $V$ operator as an operator on modular forms: Suppose that $f\in M_k(\SL_2(\Z),\C)$. Then $(f\mid V)(z) = f(pz)$, and as is well-known $f\mid V \in M_k(\Gamma_0(p);\C)$. The proof of the next lemma is a simple application of section 3.2 of \cite{serre_zeta}.

\begin{lem}\label{lemma:V-operator} Let $f\in M_k(\SL_2(\Z),\Q)$ and suppose $v_p(f) = 0$. Let $t\in\N$ and suppose that $s\in \Z_{\ge 0}$ is such that
$$
\inf (s+1,p^s+1-k) \ge t .
$$

Then there is $h\in M_{k+p^s(p-1)}(\SL_2(\Z),\Q)$ with $v_p(h) = 0$, and such that
$$
f\mid V \equiv h \pmod{p^t}.
$$
\end{lem}
\begin{proof} As we noted above, $f\mid V$ is a modular form of weight $k$ on $\Gamma_0(p)$. Since $f\mid V = \sum a_n q^{np}$ if $f=\sum a_n q^n$ we have $v_p(f\mid V) = 0$. Recall the Fricke involution for modular forms on $\Gamma_0(p)$ given by the action of the matrix
$$
W = \abcd{0}{-1}{p}{0} .
$$

Since
$$
f\mid V = p^{-k/2} f\mid_k \abcd{p}{0}{0}{1} ,
$$
(recall that the weight $k$ action is normalized so that diagonal matrices act trivially), since
$$
\abcd{p}{0}{0}{1} \abcd{0}{-1}{p}{0} = \abcd{0}{-1}{1}{0} \abcd{p}{0}{0}{p} ,
$$
and since $f$ is on $\SL_2(\Z)$ we see that $f\mid VW = p^{-k/2} f$ so that $v_p(f\mid VW) = -k/2$.

Now let $E:=E_{p-1}$ and put
$$
g := E - p^{p-1} (E\mid V)
$$
so that $g$ is a modular form of weight $p-1$ on $\Gamma_0(p)$. Then, if we put
$$
f_s := \Tr ((f\mid V)\cdot g^{p^s})
$$
for $s\in\Z_{\ge 0}$ where $\Tr$ denotes the trace from $\Gamma_0(p)$ to $\SL_2(\Z)$, it follows from section 3.2 of \cite{serre_zeta} that $f_s$ is a modular form on $\SL_2(\Z)$ of weight $k+p^s(p-1)$ and rational $q$-expansion. Furthermore, Lemme 9 of \cite{serre_zeta} implies that
\begin{eqnarray*}
v_p(f_s-(f\mid V)) & \ge & \inf (s+1,p^s+1+v_p(f\mid VW) -k/2) \\
& = & \inf (s+1,p^s+1-k) \ge t .
\end{eqnarray*}

Hence, we can choose $h:=f_s$. As $f\mid V \equiv h \pmod{p^t}$ and $v_p(f\mid V) = 0$, we must have $v_p(h) = 0$.
\end{proof}

\begin{proof}[Proof of Proposition \ref{prop:G_2_modpm}:] That the defined weights $k_0,\ldots,k_{m-1}$ satisfy $k_0 < \ldots < k_{m-1}$ is verified immediately.

Thus, starting with Corollary \ref{cor:G_2a} we see that it suffices to show that for each $j\in \{ 0,\ldots ,m-1\}$ there is a modular form $f_j$ of weight $k_j$ with rational $q$-expansion and $v_p(f_j) = 0$, and such that
$$
G_{2+p^{m-j-1}(p-1)} \mid V^j \equiv f_j \pmod{p^{m-j}} .
$$

If $m=1$, $j=0$ we just take $f_0 := G_{p+1}$, so assume $m\ge 2$. Then, if $j=m-1$, note that
$$
G_{p+1} \mid V^{m-1} \equiv G_{p+1}^{p^{m-1}} \pmod{p}
$$
so that we can take $f_{m-1} := G_{p+1}^{p^{m-1}}$ that is of weight $k_{m-1} = p^{m-1}(p+1)$.

So, suppose that $j\le m-2$. We claim that for $r=0,\ldots ,j$ there is a modular form $g_r$ of weight $2+p^{m-j-1}(p^{r+1}-1)$, rational $q$-expansion with $v_p(g_r)=0$, and such that
$$
G_{2+p^{m-j-1}(p-1)} \mid V^r \equiv g_r \pmod{p^{m-j}}
$$
which is the desired result when $r=j$.

We prove the last claim by induction on $r$ noting that the case $r=0$ is trivial. So, suppose that $r<j$ and that we have already shown the existence of a modular form $g_r$ as above. Notice that
$$
p^{m-j+r} + 1 - (2 + p^{m-j-1}(p^{r+1}-1)) = p^{m-j-1} - 1 \ge m-j
$$
holds because $m-j \ge 2$ (we used here that $p>2$). Thus we see that Lemma \ref{lemma:V-operator} applies (taking $s = m-j+r$) and shows the existence of a modular form $g_{r+1}$ with rational $q$-expansion and $v_p(g_{r+1})=0$, such that
$$
g_r \mid V \equiv g_{r+1} \pmod{p^{m-j}},
$$
and such that $g_{r+1}$ has weight
$$
2+p^{m-j-1}(p^{r+1}-1) + p^{m-j+r}(p-1) = 2+p^{m-j-1}(p^{r+2}-1),
$$
and we are done.
\end{proof}

\begin{remark}
It is interesting to note that in the induction, the inequalities do not allow us to deal with the case $j = m-1$, but then we use the congruence $f \mid V \equiv f^p \pmod p$ to take care of the last term.
\end{remark}

\begin{proof}[Proof of Proposition \ref{prop:G_2_modp2}:] Again by Corollary \ref{cor:G_2a} we have:
$$
G_2 \equiv G_{2+p(p-1)} + p\cdot \left( G_{p+1} \mid V \right) \pmod{p^2} .
$$

Noting again that $G_{p+1} \mid V \equiv G_{p+1}^p \pmod{p}$ so that
$$
p\cdot \left( G_{p+1} \mid V \right) \equiv p\cdot G_{p+1}^p \pmod{p^2} ,
$$
we are done.
\end{proof}

\subsection{Definition and properties of the \texorpdfstring{$\theta$}{theta} operator} \label{def_theta}

Recall the classical $\theta$ ope\-rator acting on formal $q$-expansions as $q\frac{d}{dq}$, i.e.,
$$
\theta \left( \sum a_n q^n \right) := \sum n a_n q^n,
$$
and the operator $\partial$ defined by
\begin{equation*}
\frac{1}{12} \partial f := \theta f - \frac{k}{12} E_2 \cdot f = \theta f + 2k G_2\cdot f
\end{equation*}
when $f = \sum a_n q^n \in M_k(N,\C)$ is a modular form of weight $k$ (we have $B_2 = \frac{1}{6}$ so that $E_2 = -24 G_2$). Then $\partial f$ is in $M_{k+2}(N,\C)$, and $\partial$ defines a derivation on $M(N,\C) := \oplus_k M_k(N,\C)$ as follows by writing $\theta = \frac{1}{2\pi i} \cdot \frac{d}{dz}$ (as $q = e^{2\pi i z})$ and combining with the classical transformation properties of $E_2$ under the weight $2$ action of $\SL_2(\Z)$ given below in \eqref{E-two}. Also, these facts can be used to deduce that $\partial$ preserves cusp forms.

The definition of $\partial$ implies that $\partial$ defines a derivation on $\oplus_k M_k(N,\Z)$ and hence also on $M(N,\Z_p) := \oplus_k M_k(N,\Z_p)$.

\begin{proof}[Proof of Theorem \ref{theta-single-wt}]

{\it (i)} Retain the notation of Proposition \ref{prop:G_2_modpm} so that
$$
k_j := 2 + p^{m-j-1}(p^{j+1}-1) \quad \mbox{for $j=0,\ldots,m-2$} ,
$$
and
$$
k_{m-1} := p^{m-1}(p+1) .
$$

Then $k_0 < \ldots < k_{m-1}$ and by Proposition \ref{prop:G_2_modpm} we have modular forms $f_0,\ldots,f_{m-1}$ (of level $1$ and) of weights $k_0,\ldots,k_{m-1}$, respectively, that have rational $q$-expan\-sions, satisfy $v_p(f_j) = 0$ for all $j$, and are such that
$$
G_2 \equiv \sum_{j=0}^{m-1} p^j f_j \pmod{p^m} .
$$

With $k(m) := 2 + 2 p^{m-1}(p-1)$ one checks that each number $k(m) - k_j$ is a multiple of $p^{m-j-1}(p-1)$, say $k(m) = k_j + t_j\cdot p^{m-j-1}(p-1)$ for $j=0,\ldots, m-1$.

Since $E_{p-1}^{p^{m-j-1}} \equiv 1 \pmod{p^{m-j}}$ we find that $p^j E_{p-1}^{p^{m-j-1} t_j} \equiv p^j \pmod{p^m}$, and so the above congruence for $G_2$ can also be written as
$$
G_2 \equiv \sum_{j=0}^{m-1} p^j E_{p-1}^{p^{m-j-1} t_j} f_j \pmod{p^m}
$$
where now each summand is a form of weight $k(m)$.

Hence, for an element $f\in M_k(N,\Z_p)$ we find that
\begin{align}
\theta f & = \frac{1}{12} \partial f - 2k G_2\cdot f \\
\label{theta-defn}
& \equiv \frac{1}{12} E_{p-1}^{2p^{m-1}} \partial f - 2k f
\sum_{j=0}^{m-1} p^j E_{p-1}^{p^{m-j-1} t_j} f_j \pmod{p^m} \\
\label{theta-defn1}
& =: \theta_{p^m} f \in M_{k+k(m)}(\Z_p)
\end{align}
where now each summand on the right hand side is an element of $M_{k+k(m)}(N,\Z_p)$. Thus the classical theta operator induces a linear map
$$
\theta = \theta_{p^m} : M_k(N,\Zmodpm{p^m}) \rightarrow M_{k + k(m)}(N,\Zmodpm{p^m})
$$
the effect of which on $q$-expansions is $\sum a_n q^n \mapsto \sum
n a_n q^n$. We still denote this operator (by abuse of notation) by
$\theta$, but later when we need to distinguish it from $\theta := \frac{1}{2\pi i} \cdot \frac{d}{dz}$, we will denote it by $\theta_{p^m}$. Since, as noted above, the operator $\partial$ preserves cusp forms we see from the definition that $\theta_{p^m}$ does too.
\smallskip

\noindent {\it (ii)} First assume for the prime $\ell$ that we have $\ell\nmid N$. Recall that the diamond operator $\diam{\ell}_k$ on a modular form $f$ of weight $k$ is defined by $\diam{\ell}_k f = f \mid_k \gamma$ for any $\gamma = \begin{pmatrix} a & b \\ c & d \end{pmatrix} \in \SL_2(\Z)$, with $c \equiv 0 \pmod N$, $d \equiv \ell \pmod N$. (Note that we write the action of $\diam{\cdot}_k$ from the left, though it is given by the stroke operator which is from the right. This is fine because $(\Z/N\Z)^\times$ is abelian). Now, if $\ell \neq p$ then the operator $\diam{\ell}_k$ induces a linear action on $M_k(N,\Z_p)$ (as follows from the well-known formula $\ell^{k-1} \diam{\ell}_k = T_{\ell}^2 - T_{\ell^2}$ and the fact that the Hecke operators $T_n$ preserve $M_k(N,\Z)$), and hence also on $M_k(N,\Zmodpm{p^m})$. This is still true when $\ell = p$: for, since $p\nmid N$ and since the definition of $\diam{\ell}_k$ only depends on the congruence class of $\ell$ mod $N$, we can choose a prime $\ell_1$ not dividing $Np$ such that $\diam{\ell_1}_k = \diam{p}_k$, and the claim follows from what has already been said.
\smallskip

Using $\frac{1}{12} \partial f = \theta f -\frac{k}{12} E_2 f$ as well as the transformation property of $E_2$ given by
\begin{equation}
\label{E-two}
(E_2 \mid_2 \gamma)(z) = E_2(z) + \frac{12}{2\pi i} c (cz+d)^{-1}
\end{equation}
for $\gamma = \begin{pmatrix} a & b \\ c & d \end{pmatrix} \in \SL_2(\Z)$ (see for instance \cite{ds}), a computation shows that we have $\frac{1}{12} \partial \diam{\ell}_k f = \diam{\ell}_{k+2} \frac{1}{12} \partial f$ for $f \in M_k(N,\C)$.

Now recall the above definition of $\theta_{p^m} : M_k(N,\Zmodpm{p^m}) \rightarrow M_{k + k(m)}(N,\Zmodpm{p^m})$, as well as the fact that the forms $E_{p-1}$ and the $f_j$ occurring in the definition all have level $1$ and thus are fixed under the action of the operators $\diam{\ell}_{p-1}$ and $\diam{\ell}_{k_j}$, respectively. We deduce that:
$$
\theta_{p^m} \diam{\ell}_k f = \diam{\ell}_{k+k(m)} \theta_{p^m} f
$$
for all $f \in M_k(N,\Zmodpm{p^m})$.

Let now $f = \sum a_n q^n \in M_k(N,\Zmodpm{p^m})$. Then we have (with the usual convention that $a_{\frac{n}{\ell}} := 0$ if $\ell \nmid n$):
\begin{eqnarray*}
  \diam{\ell}_k f &=& \sum b_n q^n \\
  T_\ell f &=&  \sum \left( a_{\ell n} + \ell^{k-1} b_{\frac{n}{\ell}} \right) q^n \\
  \diam{\ell}_{k+k(m)} \theta_{p^m} f &=& \theta_{p^m} \diam{\ell}_k f = \sum n b_n q^n \\
  T_\ell \theta_{p^m} f &=& \sum \left( \ell n a_{\ell n} + \ell^{k+k(m)-1} \frac{n}{\ell} b_{\frac{n}{\ell}} \right) q^n \\
  &=&  \sum \ell n \left( a_{\ell n} + \ell^{(k-1) + k(m)-2} b_{\frac{n}{\ell}} \right) q^n \\
  \ell \theta_{p^m} T_\ell f &=& \sum \ell n \left( a_{\ell n} + \ell^{k-1} b_{\frac{n}{\ell}} \right) q^n.
\end{eqnarray*}
For $\ell \mid N$ a similar calculation holds (the second term involving $b_{\frac{n}{\ell}}$ is omitted throughout).

Thus, if $\ell \neq p$ we have that
$$
T_\ell \theta_{p^m} f = \ell \theta_{p^m} T_\ell f
$$
for all $f \in M_k(N,\Zmodpm{p^m})$ using the fact that $k(m)-2 = 2p^{m-1}(p-1)$ is divisible by $p^{m-1}(p-1)$ so that
$$
\ell^{k(m)-2} \equiv 1 \pmod{p^m}.
$$

Finally, if $\ell = p$ and $k \ge m$, we still have the desired identity as the discrepancy between the right hand side and left hand side is divisible by $\ell n \cdot \ell^{k-1}$ and hence is congruent to $0$ modulo $p^m$.

\noindent {\it (iii)} Now suppose that $m\ge 2$, that $f \in M_k(N,\Zmodpm{p^m})$ with $f \not\equiv 0 \pmod p$, and suppose further that $p\nmid k$ and $w_p(f)= k$. These hypotheses imply that $w_p(\theta f) = k + p + 1$ whence in particular that $\theta f \not\equiv 0 \pmod{p}$ (cf.\ \cite{serre_modp} and \cite{swinnerton-dyer} for level one, \cite{katz_modp} for higher levels).

Assume that we had $w_{p^m}(\theta f) < k + 2 + 2p^{m-1}(p-1) = k+k(m)$, i.e., that there exists a form $g\in M_{k'}(N,\Zmodpm{p^m})$ with $g=\theta f$ as forms with coefficients in $\Zmodpm{p^m}$, and where $k'<k + k(m)$. As $\theta f \not\equiv 0 \pmod{p}$ we can then deduce that
$$
k' \equiv k + k(m) \pmod{p^{m-1}(p-1)}
$$
by classical results due to Serre and Katz, cf.\ specifically Corollary 4.4.2 of \cite{katz_padic}. We use the fact that modular forms in $M_k(N,\Z/p^m\Z)$ can be lifted to classical modular forms over $\Z_p$ (and can in fact be lifted to $\Z$) and that $N$ is prime to $p$. For a generalization that applies in situations where the forms have more general coefficients, see \cite[Theorem B]{queen}. This generalization is needed for the proof of Corollary \ref{main-cor} since that corollary applies to forms with coefficients in $\Zmod{p^m}$.

Hence, let us write $k+k(m) = k'+t\cdot p^{m-1}(p-1)$ with $t\ge 1$. Putting
$$
h := E_{p-1}^{p^{m-1}(t-1)} g
$$
we find that
$$
\theta f = E_{p-1}^{p^{m-1}} h
$$
as an equality of forms in $M_{k+k(m)}(N,\Zmodpm{p^m})$. If we combine this with \eqref{theta-defn} and \eqref{theta-defn1} above, we obtain:
$$
2kp^{m-1} E_{p-1}^{t_{m-1}} f_{m-1} f = -E_{p-1}^{p^{m-1}} h + \frac{1}{12} E_{p-1}^{2p^{m-1}} \partial f - 2k f \sum_{j=0}^{m-2} p^j E_{p-1}^{p^{m-j-1} t_j} f_j .
$$

If we now use the fact that $p\nmid k$, that $p$ is odd, and that, as is easily checked, we have
$$
t_{m-1} = (k(m)-k_{m-1})/(p-1)  = (p^m - 3 p^{m-1} + 2)/(p-1)< p^{m-1} ,
$$
as well as $t_{m-1} < p^{m-j-1} t_j$ for $j=0,\ldots, m-2$, we deduce that
$$
p^{m-1} E_{p-1}^{t_{m-1}} f_{m-1} f = E_{p-1}^{t_{m-1}+1} h'
$$
for some $h' \in M_{k+k(m)-(p-1)(t_{m-1}+1)}(N,\Zmodpm{p^m})$. Hence we must have
$$
h' \in p^{m-1} M_{k+k(m)-(p-1)(t_{m-1}+1)}(N,\Zmodpm{p^m}) ,
$$
say $h'=p^{m-1}h''$, so that
$$
E_{p-1}^{t_{m-1}} f_{m-1} f \equiv E_{p-1}^{t_{m-1}+1} h'' \pmod{p} ,
$$
and hence
$$
f_{m-1} f \equiv E_{p-1} h'' \pmod{p} .
$$

It follows that
$$
w_p(f_{m-1} f) < k+k(m) - t_{m-1}(p-1) = k+k_{m-1} = k+p^{m-1}(p+1) .
$$

Now recall (from Proposition \ref{prop:G_2_modpm}) that $f_{m-1} = G_{p+1}^{p^{m-1}}$. As $G_{p+1} = -\frac{B_{p+1}}{2(p+1)} E_{p+1}$ with $\frac{B_{p+1}}{2(p+1)}$ invertible modulo $p$, we deduce
$$
w_p(E_{p+1}^{p^{m-1}} f) < k+p^{m-1}(p+1) .
$$

However, as $w_p(f) = k$ by assumption (and as $p\nmid k$ so that, in particular, $k\neq p$), this conclusion contradicts Lemma \ref{lemma:filtration_e_p+1_phi} below.
\end{proof}

\begin{lem}\label{lemma:filtration_e_p+1_phi} Suppose that $p\neq \kappa\in\N$ and that $0\neq \phi \in M_\kappa(N,\Z/p\Z)$ with $w_p(\phi) = \kappa$. Then, for $a\in \N$ we have:
$$
w_p(E_{p+1}^a \phi) = w_p(\phi) + a (p+1) .
$$
\end{lem}

\begin{proof} By induction on $a$ it is clearly enough to prove the case $a=1$. Hence, let us assume $a=1$.

Assume for a contradiction that we had $w_p(E_{p+1} \phi) < \kappa + p+1$. Then
$$
E_{p+1} \phi = E_{p-1} \psi
$$
for some $\psi \in M_{2+\kappa}(N,\Z/p\Z)$: put $k := w_p(E_{p+1} \phi)$, and let $g\in S_k(N,\Z/p\Z)$ be such that $g$ and $E_{p+1} \phi$ have identical $q$-expansions. Then by the classical result of Serre and Katz (Corollary 4.4.2 of \cite{katz_padic}) that we already used above, it follows that $k \equiv \kappa + p+1 \pmod{p-1}$. Writing $\kappa + p+1 = k + t\cdot (p-1)$ with $t\in \N$ we also have $E_{p+1} \phi = g\cdot E_{p-1}^t$. Putting $\psi = g\cdot E_{p-1}^{t-1}$ we have the claim.
\smallskip

Now suppose first that $N \ge 5$. Let $\mathcal{M}_k(N,R)$ denote the space of modular forms of weight $k$ on $\Gamma_1(N)$ over $R$ as defined in \cite{di} using Katz's definition. One has an injection
$$
M_k(N,R) \rightarrow \mathcal{M}_k(N,R)
$$
sending classical modular forms over $R$ to Katz modular forms over $R$.

Under the hypothesis $N \ge 5$, we have from \cite[Theorem 12.3.7]{di}) that:
\begin{enumerate}
\item[(K)] $\mathcal{M}_k(N,\Z/p\Z) \cong M_k(N,\Z/p\Z)$ if $k \neq 1$ and $N$ is prime to $p$.
\end{enumerate}

We note that $k=0$ is allowed in (K), and we will need that below.



Regard the above identity as an identity of Katz modular forms on $\Gamma_1(N)$ over $\Z/p\Z$ and let us use some results from \cite{katz_modp}: By the remark after Lemma 1 of \cite{katz_modp}, the forms $E_{p-1}$ and $E_{p+1}$ are without a common zero (the results in \cite{katz_modp} are for modular forms on $\Gamma(N)$ which implies the result for $\Gamma_1(N)$. This result is also proved in \cite[Theorem 3.1]{katz_zannier} directly for $E_{p-1}$ and $E_{p+1}$ as forms on $\SL_2(\Z)$.)

Hence the equality $E_{p+1} \phi = E_{p-1} \psi$ means that $\phi$ vanishes at every zero of $E_{p-1}$ to at least the order that $E_{p-1}$ vanishes to at that zero. (In fact, by a theorem of Igusa \cite{igusa}, $E_{p-1}$ vanishes to order $1$ at every supersingular point of $X_1(N)$, see \cite[second paragraph following (4.6)]{gross}, but this additional information is irrelevant for our deduction).

Thus, we must have $\kappa \ge p-1$ and $\phi = E_{p-1} \eta$ for some $\eta \in \mathcal{M}_{\kappa-(p-1)}(N,\Z/p\Z)$.

By hypothesis we have $\kappa-(p-1) \neq 1$ and so by (K) above we have that $\eta$ is classical in the sense that $\eta \in M_{\kappa-(p-1)}(N,\Z/p\Z)$. But then $w_p(\phi) \le \kappa - (p-1)$, contrary to hypothesis.
\smallskip

Suppose then that $N\le 4$. We can reduce to the previous case as follows. Choose an auxiliary prime $q\ge 5$ such that $q \nmid Np$ and such that $p \nmid t:= q^2-1$, and regard $\phi$ as a Katz modular form on $\Gamma_1(Nq)$. The above argument then shows that under the assumption $w_p(E_{p+1}\phi) < \kappa + p+1$ we will have $\phi = E_{p-1} \eta$ for some $\eta \in M_{\kappa-(p-1)}(Nq,\Z/p\Z)$.

Choose forms $\tilde\phi \in M_{\kappa}(N,\Z_p)$ and $\tilde\eta \in M_{\kappa-(p-1)}(Nq,\Z_p)$ that reduce to $\phi$ and $\eta$, respectively. Also, let $\zeta$ be a primitive $Nq$'th root of unity, and let $O$ be the ring of integers of $\Q(\zeta)$ completed at a prime $\p$ above $p$.

Let $\gamma_1,\ldots,\gamma_t$ be a set of right coset representatives for $\Gamma_1(Nq)$ in $\Gamma_1(N)$. We then have the trace map
$$
f \mapsto \Tr f := \sum_{i=1}^t f\mid_k \gamma_i
$$
as a linear map $M_k(Nq,\C) \rightarrow M_k(N,\C)$ for any $k$. However, $\Tr$ as defined above in fact also defines a linear map $M_k(Nq,O) \rightarrow M_k(N,O)$: This follows as we know that if $f \in M_k(Nq,O)$ is a form for which the $q$-expansion (at $\infty$) has coefficients in $O$ then the same is true for $f\mid_k \gamma$ for any $\gamma\in\SL_2(\Z)$. By our definition of $M_k(Nq,O)$, it suffices to know that if $f \in M_k(\Gamma(Nq),\C)$ has $q$-expansion with coefficients in $\Z[1/n][\zeta]$, then $f \mid_k \gamma$ also has the same property for any $\gamma \in \SL_2(\Z)$. This latter fact is stated in \cite[Chap.\ VII, \S 4.8]{delrap} as a consequence of \cite[Chap.\ VII, Corollaire 3.13]{delrap}.

Taking the trace we find the congruence $t\tilde \phi \equiv E_{p-1} (\Tr \tilde \eta) \pmod{\p}$ as a congruence between forms in $M_{\kappa}(N,O)$ (we used that $\tilde \phi$ and $E_{p-1}$ are on $\Gamma_1(N)$). Since $t = [\Gamma_1(N) : \Gamma_1(Nq)] = q^2-1$ is prime to $p$, the congruence shows that $\phi$ coincides with an element in $M_{\kappa-(p-1)}(N,O/\p)$. However, because $M_{\kappa-(p-1)}(N,\Z/p\Z) \otimes O/\p \cong M_{\kappa-(p-1)}(N,O/\p)$ and $\phi \in M_{\kappa-(p-1)}(N,O/\p)$ has $q$-expansion lying in $\Z/p\Z[[q]]$ and $O/\p$ is free over $\Z/p\Z$, it follows that $\phi$ in fact lies in $M_{\kappa-(p-1)}(N,\Z/p\Z)$. But this contradicts $w_p(\phi) = \kappa$.
\end{proof}


\begin{remark}
\label{remark:cusp_vs_modular}
We have used an ad hoc argument to deduce (K) for $N \le 4$ and $p \ge 5$ from a property of $q$-expansions of classical forms on the principal congruence groups (i.e.\ \cite[Chap.\ VII, \S 4.8]{delrap}), which is in turn a consequence of the fact that classical forms map to Katz forms in a way that is compatible with $q$-expansions and actions by $\SL_2(\Z)$ and then noting the result for Katz forms. The ad hoc argument we use is elementary and relies on an a priori weaker statement. It is possible to give a more direct and conceptual argument, for instance by using the ingredients in \cite[Theorem 1.7.1]{katz_modp} or \cite[Lemma 1.9]{edix} with auxiliary full level $n = 3$ and taking $G = \GL_2(\Z/3\Z)$-invariants (so $\#G$ is prime to $p \ge 5$), but proving the exact form of what we need in this fashion requires further elaboration.
\end{remark}

\section{Acknowledgements} We would like to thank Nadim Rustom for a useful discussion about an early version of Lemma \ref{lemma:filtration_e_p+1_phi}. We would also like to thank the anonymous referee for numerous comments and suggestions that improved the paper. I.K.\ acknowledges support by grants from the Danish Council for Independent Research, and from VILLUM FONDEN through the network for Experimental Mathematics in Number Theory, Operator Algebras, and Topology. I.C.\ acknowledges support from NSERC.


\end{document}